\documentclass[submission,copyright,creativecommons]{eptcs}
\providecommand{\event}{QPL 2018}

\usepackage{amsmath}
\usepackage{amsfonts}
\usepackage{xspace}
\usepackage{amsthm}
\usepackage{physics}
\usepackage{mathtools}

\usepackage{tikz}
\usepackage{tikz-cd}
\usetikzlibrary{arrows, matrix}

\newcounter{main}
\newtheorem{theorem}[main]{Theorem}
\newtheorem{proposition}[main]{Proposition}
\newtheorem{lemma}[main]{Lemma}
\newtheorem{corollary}[main]{Corollary}

\theoremstyle{definition}
\newtheorem{definition}[main]{Definition}
\newtheorem{notation}[main]{Notation}
\newtheorem{note}[main]{Note}
\newtheorem{example}[main]{Example}

\DeclarePairedDelimiter{\ceil}{\lceil}{\rceil}
\DeclarePairedDelimiter{\floor}{\lfloor}{\rfloor}
\DeclarePairedDelimiter{\inn}{\langle}{\rangle}

\newcommand{\cl}[1]{\overline{#1}}
\DeclareMathOperator{\im}{im}

\newcommand{\EJA}{\text{\textbf{EJA}}\xspace}

\newcommand{\id}{\mathrm{id}}
\newcommand{\opp}{\mathrm{op}}
\newcommand{\sa}{\mathrm{sa}}

\newcommand\R{\mathbb{R}}
\newcommand\N{\mathbb{N}}
\newcommand\after{\mathbin{\circ}}
\DeclareMathOperator{\Idem}{Idem}

\def\titlerunning{Pure Maps between Euclidean Jordan Algebras}
\title{\titlerunning}

\author{Abraham Westerbaan
\institute{Radboud Universiteit}
\email{bram@westerbaan.name} \and
    Bas Westerbaan
\institute{Radboud Universiteit}
\email{bas@westerbaan.name} \and
    John van de Wetering
\institute{Radboud Universiteit}
\email{wetering@cs.ru.nl}
}
\def\authorrunning{Westerbaan, Westerbaan \& van de Wetering}

\begin{document}

\maketitle

\begin{abstract}
We propose a definition of purity for positive linear maps between 
Euclidean Jordan Algebras (EJA) that generalizes the notion of 
purity for quantum systems. 
We show that this definition of purity is closed under composition
and taking adjoints and thus that the pure maps form a dagger category
(which sets it apart from other possible definitions.)
In fact, from the results presented in this paper, it follows that
    the category of EJAs with positive contractive
linear maps is a \emph{$\dagger$-effectus}, a type of structure originally
defined to study von Neumann algebras in an abstract categorical setting.
In combination with previous work this
characterizes EJAs as the most general systems allowed in a 
generalized probabilistic theory that is simultaneously a $\dagger$-effectus.
Using the dagger structure we get a notion of $\dagger$-positive maps 
of the form~$f = g^\dagger \circ g$. We give a complete characterization of
the pure $\dagger$-positive maps and show that these correspond precisely to
the Jordan algebraic version of the sequential product~$(a,b)\mapsto \sqrt{a}b\sqrt{a}$. 
The notion of $\dagger$-positivity therefore characterizes the sequential product.
\end{abstract}

\section{Introduction}
A commonly used technique when studying the foundations of quantum theory, is to consider generalized theories that only exhibit some part of the features of conventional quantum theory. 
In this way, it becomes more clear what specific properties of quantum theory lead to certain structure. 
One of the first generalized quantum theories to be studied were the Euclidean Jordan algebras (EJAs)~\cite{jordan1933}. 
Besides the matrix algebras of complex self-adjoint matrices of conventional quantum theory, other examples of EJAs are the set of real symmetric matrices of real-valued quantum theory, or the set of self-adjoint matrices over the quaternions. 
Quite soon after the introduction of EJAs, a full characterization of EJAs was given~\cite{jordan1993algebraic} that showed that these examples almost completely exhaust the possibilities. 
The Koecher--Vinberg theorem~\cite{koecher1957positivitatsbereiche} is a major result stating that any ordered vector space with a homogeneous self-dual positive cone is a Euclidean Jordan algebra. 
It is this theorem that explains the ubiquity of EJAs in reconstructions of quantum theory~\cite{barnum2014higher,wetering2018reconstruction,masanes2014entanglement,wilce2016royal,selby2018reconstructing,wetering2018sequential}.
Understanding the differences and similarities between regular quantum theory as described by complex matrix algebras, and the more general EJAs is an active topic of research. In this paper we will study the notion of \emph{pure maps} in EJAs and show that they have many of the same properties as those found in quantum theory.

The concept of purity has proven very useful in the field of quantum
information. In the context of states, it can be considered a
resource in various protocols and computations~\cite{devetak2005distillation,brandao2013resource} and the
possibility of purification of states is considered to be one of
the characteristic features differentiating quantum theory from its
classical counterpart~\cite{chiribella2010probabilistic,chiribella2011informational}.
While there is a generally accepted definition of purity for states,
when it comes to quantum channels there are several proposed definitions of
purity in play, each of which has its drawbacks.
There is for instance the definition of \emph{atomicity}
used in reconstructions of quantum theory~\cite{chiribella2011informational}.
This definition is very general in that it can be defined for any generalized probabilistic theory~\cite{barrett2007information},
but it has the drawback that even a canonical map like the identity will not always be pure.
Purity can also be defined in terms of leaks~\cite{selby2018reconstructing},
or using orthogonal factorization~\cite{cunningham2017purity}.
These definitions work well when considering finite dimensional spaces with a well-behaved sequential product,
but when considering more general quantum systems like von Neumann algebras, they fail to reproduce many desirable properties.
More fundamentally,
    both definitions require a notion of tensor product,
    the existence of which for EJAs is an open question~\cite{barnum2016composites}.
Instead we will use a definition originally introduced for studying pure maps between von Neumann algebras~\cite{bramthesis,westerbaan2016paschke,westerbaan2016universal},
and that also defines the main concept of purity in \emph{effectus theory}~\cite{cho2015introduction,basthesis}.
While the other definitions of purity are related to the existence of \emph{purifications} of states, 
this definition is related to the existence of \emph{Paschke dilations}~\cite{westerbaan2016paschke}.
Following this definition, a map is pure when it is a composition of a \emph{filter} and a \emph{corner}. 
Filters represent a certain restricted class of measurements
    (named after polarization filters).
Corners correspond to the act of restricting a system to a subsystem.
Filters and corners are defined by universal properties,
    and thus this definition of purity makes sense even in the very
    general setting of effectus theory.
In effectus theory, filters and corners are known as \emph{quotients}
and \emph{comprehensions} respectively, since they correspond to those logical
operations (with the reversed direction of maps).
Quotients and comprehension appear in a multitude of categories
    (outside effectus theory) as a chain of adjunctions \cite{cho2015quotient}.

In this paper we show that Euclidean Jordan algebras indeed
have pure maps as described above, and furthermore that the
pure maps can be organised in a dagger category. This is
noteworthy for several reasons.
First, the other definitions of purity on
quantum systems are in general not closed under composition,
or taking the adjoint. Hence this shows that this definition
of purity seems to be the right one for EJAs.
Second, that the pure maps form a dagger category, which is the main result
needed to show that the category $\EJA_{psu}$ of EJAs with positive subunital
maps between them is a \emph{$\dagger$-effectus}~\cite{basthesis}.
The notion of $\dagger$-effectus was introduced to give an abstract
categorical framework for studying von~Neumann algebras, and in fact
the only known non-trivial example of a $\dagger$-effectus was
the category of von~Neumann algebras. This paper therefore 
establishes that $\EJA_{psu}$ is indeed quite similar to 
the category of von~Neumann algebras.
Finally, it was shown in previous work by the third 
author~\cite{wetering2018reconstruction} that a
generalised probabilistic theory which is also a $\dagger$-effectus,
must be a subcategory of $\EJA_{psu}$. Hence, 
this paper proves the converse result, establishing a 
novel characterization of the category of Euclidean Jordan algebras.

To prove that the composition of pure maps is pure,
    we have found a new generalization
    of the \emph{polar decomposition theorem} to EJAs,
    which  might be of independent interest.

Motivated by effectus theory, we also undertake a study of the
\emph{possibilistic} structure of the maps between EJAs.
This means that we do not look at the actual probabilities, but
instead only consider whether probabilities are nonzero.
In this way we define a duality between maps that we call
\emph{$\diamond$-adjointness}, that generalises
the regular adjoints present by the Hilbert space structure of EJAs. 
We can similarly also define $\diamond$-self-adjointness and $\diamond$-positivity. 
We completely characterize the pure $\diamond$-positive maps and show that they exactly correspond to a generalization 
of the sequential product maps $b\mapsto \sqrt{a}b\sqrt{a}$ in von Neumann algebras, (and we'll deduce from this 
that pure $\diamond$-positivity and $\dagger$-positivity coincide.)
This result can be seen as a characterization of the sequential product like the ones given in~\cite{gudder2008characterization,wetering2018characterisation,westerbaan2016universal}.

The paper is structured as follows. In the next section we review
some of the basic structure present in Euclidean Jordan algebras.
We also refer the reader to the appendix for a concise largely self-contained
introduction to EJAs. In section \ref{sec:filtercorner} we introduce
the notion of \emph{filters} and \emph{corners} and study our
proposed notion of purity. Finally in section \ref{sec:diamond} we
use the possibilistic $\diamond$ structure to characterize the
sequential product maps as the unique pure $\diamond$-positive maps.

\section{Preliminaries}\label{sec:eja}
We begin by giving a definition of Euclidean Jordan algebras and
some motivating examples. Afterwards we will review some of the
basic theory necessary to develop our results.

\begin{definition}
	A \emph{Jordan algebra} $(E,*,1)$ is a real unital commutative
	(possibly non-associative) algebra satisfying the \emph{Jordan
	identity}: $(a*b)*(a*a) = a*(b*(a*a))$.
    We call a Jordan Algebra~$E$ with 
       an inner product $\left<\,\cdot\,,\,\cdot\,\right>$
        a \emph{Euclidean Jordan algebra} (EJA)
    when the inner product turns~$E$ into a real Hilbert space
	with
	$\inn{a*b,c}=\inn{b,a*c}$ for all~$a,b,c\in E$.
\end{definition}
\begin{note}
In the original (and frequently used) definition,
    one additionally requires a Euclidean Jordan algebra to be
    finite dimensional.
The possibly infinite-dimensional version we use
    is also called a JH-algebra \cite{chu2017infinite}, although we additionally require unitality.
\end{note}
\begin{example}
	Let $F$ be the field of real numbers, the field of complex numbers, or the division algebra of the quaternions. Let $A\in M_n(F)$ be an $n\times n$ matrix over $F$. We call $A$ \emph{self-adjoint}
    when~$A_{ij} = \cl{A_{ji}}$ where~$\cl{\lambda}$ denotes the
    standard involution on~$F$. We let the set of self-adjoint
    matrices be denoted by~$M_n(F)^\sa$. This set is an Euclidean Jordan Algebra
    with the Jordan product~$A*B:= \frac{1}{2}(AB+BA)$,
    inner product~$\inn{A,B}:= \tr(AB)$
    and identity matrix as unit.

    If~$n=3$ and~$F$ is the algebra of octonions,
        then the algebra~$M_3(F)^\sa$ is also
        a Euclidean Jordan algebra,
        which is called \emph{exceptional}.
\end{example}

\begin{example}
For any real Hilbert space $H$ (possibly infinite-dimensional),
    the set~$E:=H\oplus \R$
    is a Euclidean Jordan Algebra
    with~$(a,t)*(b,s):= (sa+tb, \inn{a,b}+ts)$
    and~$\inn{(a,t),(b,s)} = \inn{a,b}+ts$.
    Such EJAs are called \emph{spin-factors}.
\end{example}

The Jordan--von Neumann--Wigner classification theorem~\cite{jordan1993algebraic}
asserts that any finite-dimensional EJA is the direct sum of the
finite-dimensional examples given above. This statement is still
true for our possibly infinite-dimensional class of EJAs, but then
the spin-factors must also be allowed to be infinite-dimensional
(a proof of this fact can be found in the appendix as
corollary~\ref{appendixjnw}).

\begin{definition}
    Let $E$ be an EJA. We call $a\in E$ \emph{positive}
    (and write $a\geq 0$) when there exists a~$b\in E$
    such that~$a=b*b$. We write $a\geq c$ when $a-c\geq 0$.
    We call a linear map between EJAs $f\colon E\rightarrow F$ \emph{positive}
    when~$f(a) \geq 0$ for all~$a \geq 0$.
    A map~$f$ is called~\emph{unital} whenever~$1_F=f(1_E)$ and
    \emph{subunital} provided~$f(1_E)\leq 1_F$.
	We let $\EJA_{psu}$ denote the category of Euclidean Jordan
	algebras with positive linear subunital maps between them.
\end{definition}

For the matrix algebras the definition of positivity coincides with
the regular definition of a positive matrix. For a spin
factor~$H\oplus \R$ we have $(a,t)\geq 0$ iff $t\geq \norm{a}_2$, so that
the set of positive elements forms the positive light-cone in a
Lorentzian space-time.

\begin{definition}
    A positive unital linear  map~$\omega\colon E\rightarrow \R$
        is called a \emph{state} on~$E$.
        An \emph{effect} is a positive
	subunital linear map $a\colon \R\rightarrow E$ and corresponds
        to an~$a\in E$ with $0\leq a\leq 1$.
    We will also call such an element~$0 \leq a \leq 1$, an effect.
\end{definition}
Because $E$ is a Hilbert space, any map $f:E\rightarrow \R$ is of
the form $f(a)=\inn{a,b}$ for some specific $b$. As we will see
that the inner product of positive elements is positive
(proposition \ref{prop:selfduality}),
any state
on an EJA will be given by a positive~$a\in E$ with $\inn{a,1}=1$.
For an effect $a$ we write $a^\perp:= 1-a$ to denote its
\emph{complement}.

We wish to define a notion of purity for the category $\EJA_{psu}$.
The canonical example of a pure map in a quantum system is the L\"uders map $B\mapsto ABA$ for some fixed matrix $A$,
and hence we would desire the generalization of these maps to EJAs to also be pure.
It also seems reasonable that any isomorphism should be pure. 
As we are working on a Hilbert space, every map has an adjoint, and it makes sense to let the adjoint of a pure map be pure again.
Finally, as we want our pure maps to be closed under composition, any composition of the examples here above should also be pure.
We will see how to make these ideas exact, and in particular how
to generalize the idea of a L\"uders map to
arbitrary EJAs. But first we need to know more about the structure
of EJAs.

\begin{theorem}\label{theor:chu}
	Let $(E,*,1)$ be an EJA. Sums of positive elements are again positive, so that the set of positive elements forms a \emph{cone}. Furthermore:
	\begin{enumerate}
		\item The unit $1$ is a \emph{strong Archimedean unit}: for all $a\in E$ there exists $n\in \N$ such that $-n1\leq a\leq n1$, and if $a\leq \frac{1}{n} 1$ for all $n\in \N$,$n\neq 0$ then $a\leq 0$.
		\item As a consequence the algebra is an \emph{order unit space}
            so that $\norm{a}:=\inf\{r\in \R~;~ -r1\leq a\leq r1\}$
            is a norm, called the \emph{order unit norm}, for which $a\leq \norm{a}1$.
		\item The order unit norm defined above, and the norm induced by the inner product are equivalent so that the topologies they induce are the same. The space is therefore also complete in the order unit norm topology.
		\item The space is \emph{self-dual}: $a\geq 0$ if and only if $\inn{a,b}\geq 0$ for all $b\geq 0$.
        \item The algebra $E$ is bounded directed complete and every state is normal so that $E$ is in fact a \emph{JBW-algebra} \cite{hanche1984jordan}.
	\end{enumerate}
\end{theorem}
\begin{proof}
    Since to our knowledge there isn't a textbook solely dedicated to infinite-dimensional EJAs we supply a relatively self-contained proof of these claims in the appendix.
\end{proof}

\begin{definition}
    For each $a\in E$ let $L_a: E\rightarrow E$ be the left-multiplication
    operator of $a$, that is: $L(a)b:= a*b$. The seemingly oddly named but important
    \emph{quadratic representation} of the Jordan algebra is
    a linear map~$Q_a\colon E \to E$ for each~$a \in E$, defined by~$Q_a
    := 2L(a)^2 - L(a^2)$.
\end{definition}
The definition of the quadratic representation~$Q_a$ might look
arbitrary, but in the case of a matrix algebra with the standard
Jordan multiplication they are exactly the L\"uders maps $Q_AB = ABA$. The
quadratic representation maps will therefore act as the Jordan
equivalent of these maps in quantum theory. The following proposition
establishes a few basic properties of the maps $Q_a$, that can be
easily checked to hold for a matrix algebra.

\begin{proposition}\label{prop:quadraticrep} For any EJA~$E$ and~$a,b,c \in E$, the following hold.
    \begin{enumerate}
        \item $Q_1 = \id$.
        \item $Q_a1 = a^2$.
        \item $\inn{Q_ab, c} = \inn{b,Q_ac}$.
        \item $Q_ab = 0 \iff Q_ba = 0 \iff a*b=0$.
        \item $Q_{Q_ab} = Q_aQ_bQ_a$ (this is known in the literature as the \emph{fundamental equality}).
        \item $Q_a$ is invertible if and only $a$ is invertible. In that case, we have $Q_a^{-1} = Q_{a^{-1}}$.
        \item $Q_a$ is a positive operator(even when $a$ is not positive). If it is invertible, it is an order automorphism. 
    \end{enumerate}
\end{proposition}
\begin{proof}
    Points 1,2 and 3 are trivial. Point 4 can be found in~\cite[Lemma
    1.26]{alfsen2012geometry}. Point 5 is usually proven using MacDonalds
    theorem (see \cite[Theorem 2.4.13]{hanche1984jordan}),
    but see also \cite{fundident}.
    Point 6 is proven in~\cite[Lemma 1.23]{alfsen2012geometry} and finally
    7 is given by \cite[Theorem 1.25]{alfsen2012geometry}.
\end{proof}

Note that while the Jordan multiplication maps $L_a$ are not positive,
and therefore are not maps in the category $\EJA_{psu}$, the quadratic
representation maps $Q_a$ are positive and when $\norm{a}\leq 1$
the maps are subunital so that they do happen to lie in $\EJA_{psu}$.

As the fundamental equality $Q_{Q_ab} = Q_aQ_bQ_a$ will be important
in the proofs below, let's unfold it for matrix algebras.
The expression $Q_{Q_AB}C$ is equal to $(ABA)C(ABA)$, while
the right-hand side of the equation $Q_AQ_BQ_AC$ is equal to
$A(B(ACA)B)A$. Thus the equality follows (in this special case) from
associativity of matrix multiplication.

\begin{definition}
	We call an element $p$ of an EJA an \emph{idempotent} iff $p^2:=p*p=p$. An idempotent is automatically positive and below the identity. We call an idempotent \emph{atomic} if there is no non-zero idempotent strictly smaller than it. Two idempotents $p$ and $q$ are \emph{orthogonal} when $p*q=0$ or equivalently $\inn{p,q}=0$.
\end{definition}

In a matrix algebra the idempotents are precisely the projections.
The quadratic representation map~$Q_p$ is then the projection map~$Q_p(a) = pap$.
The following proposition again contains claims that are easily verified to hold true in any
matrix algebra, but which for general EJAs must be proven with a bit of care.

\begin{proposition}\label{prop:EJAidempotent}
    Let $p$ be an idempotent. Then $Q_p$ is idempotent, and furthermore for $0\leq a \leq 1$
    we have
    \begin{equation*}
        Q_pa\ =\ 0 \ \iff \ \inn{a,p}\ =\ 0 \quad\text{and} \quad Q_pa \ =\  a \ \iff\  a\ \leq\ p\ \iff\  p*a\ =\ a.
    \end{equation*}
\end{proposition}
\begin{proof}
The fundamental equality (number 5 of the previous proposition)
implies that when $p$ is an idempotent: $Q_{Q_p 1}=Q_{p^2}=Q_p =
Q_pQ_1Q_p=Q_pQ_p$, so that $Q_p$ is a positive idempotent operator,
symmetric with respect to the inner product. When $Q_pa=0$ we have
$0 = \inn{Q_pa,1}=\inn{a,Q_p1} = \inn{a,p}$.
Conversely, if~$\inn{a,p} = 0$,
    then~$\inn{Q_p a, 1} = 0$.
As~$Q_p a$ is positive,
    we must have~$c^2 = Q_pa$ for some~$c$
    and so~$0 = \inn{c,c}$, hence~$c=0$ and~$Q_p a= 0$.
    Now suppose $Q_pa=a$,
then because we have $a\leq 1$ we also have by positivity of $Q_p$,
$a=Q_pa\leq Q_p1=p$. When $a\leq p$ then by definition there is a
$r\geq 0$ such that $a+r=p$. Now~$0=\inn{0,1}=\inn{p*p^\perp,1} =
\inn{p,p^\perp} = \inn{a+r,p^\perp} = \inn{a,p^\perp} + \inn{r,p^\perp}$.
By self-duality (see theorem \ref{theor:chu}) each of these terms
is positive so that we must have $0=\inn{a,p^\perp}=\inn{a,Q_{p^\perp}1}
= \inn{Q_{p^\perp}a,1}$. Since $Q_{p^\perp}a\geq 0$ this can only
be the case when $Q_{p^\perp}a = 0$ so that $p^\perp*a=0$ from which
we get $p*a = a$.
To complete the proof, assume~$p * a= a$.
We will show~$Q_p a = a$.
This follow readily from the definition and assumption: $
    Q_p a = 2 p * (p * a) - (p * p ) * a
    = 2 a - a = a $.
\end{proof}

Like in quantum theory, we have a spectral theorem for elements of
a Euclidean Jordan algebra.
\begin{proposition}
    Let $a$ be an element of an EJA. Then there exists a number~$n$,
    real numbers~$\lambda_1, \ldots, \lambda_n$
    and orthogonal atomic
    idempotents~$p_1, \ldots, p_n$
    such that $a = \sum_{i=1}^n \lambda_i p_i$.
\end{proposition}
\begin{proof}
    Proven in the appendix under Corollary~\ref{cor:spectral}.
\end{proof}

Again like in quantum theory, we can for each element consider its `range'
where it acts non-trivially. We will denote this by a `ceiling': $\ceil{a}$.
The ceiling will play an important role later when we want to restrict an
EJA to certain subspaces.

\begin{proposition}\label{prop:EJAceilfloor}
    For an effect~$a \in E$ of an EJA~$E$ (i.e.~$0 \leq a \leq 1$),
    we can find~$\lambda_i>0$ and orthogonal atomic idempotents~$p_i$
    with~$a=\sum_i\lambda_i p_i$.
With such a decomposition,
    we define~$\ceil{a}=\sum_i p_i$.
This is the least idempotent above~$a$
        (and thus independent of choice of decomposition).
We denote the de Morgan dual by~$\floor{a}=\ceil{a^\perp}^\perp$.
    This is the greatest idempotent below $a$.
\end{proposition}
\begin{proof}
Simply leaving out those terms where~$\lambda_i=0$,
    we can find~$\lambda_i\neq 0$ and orthogonal atomic idempotent~$p_i$
    with~$a = \sum_i \lambda_i p_i$.
As the inner product of positive elements is positive,
    we find~$0 \leq \inn{a , p_i} = \lambda_i \inn{p_i,p_i}$.
    As~$\inn{p_i,p_i} > 0$ we must have~$\lambda_i > 0$, as promised.

Next, we prove that~$\ceil{a}$ is the least idempotent above~$a$.
	Let $q$ be idempotent such that $a\leq q$. Then also
	$\lambda_i p_i\leq q$. By proposition \ref{prop:EJAidempotent}
	we then have $q*(\lambda_i p_i) = \lambda_i p_i$ so that
	also $q*p_i=p_i$.
    Hence~$q * \ceil{a} = \sum_i q* p_i = \sum_i p_i = \ceil{a}$.
    Again by proposition \ref{prop:EJAidempotent}
    we conclude that $\ceil{a} \leq q$.
    Thus~$\ceil{a}$ is indeed the least idempotent above~$a$.
\end{proof}

\section{Filters and Corners}\label{sec:filtercorner}
With the preliminaries out of the way we will start to look at additional structure that is present in the category $\EJA_{psu}$. The proofs in this section are heavily inspired by \cite{westerbaan2016universal,bramthesis} where the existence of this structure was shown for the category of von Neumann algebras. As stated in the introduction, our notion of purity is based on \emph{filters} and \emph{corners}. In this section we will give their formal definition and establish their existence.

\begin{definition}
	Let $q\in E$ be an effect. A \emph{corner} for $q$ is a positive subunital linear map $\pi\colon E\rightarrow \{E\lvert q\}$ such that $\pi(1)=\pi(q)$ and which is \emph{initial} with this property: if $g\colon E\rightarrow F$ is another positive subunital linear map such that $g(1)=g(q)$ then there must exist a \emph{unique} $\cl{g}\colon\{E\lvert q\}\rightarrow F$ such that $\cl{g}\circ \pi = g$. In the form of a diagram:
    \[\begin{tikzcd}[ampersand replacement = \&]
    \{E\lvert q\} \arrow[dotted,swap]{d}{\cl{g}}\&\arrow[swap]{l}{\pi} E\arrow{dl}{g} \\
     F\&  \\
    \end{tikzcd}\]
\end{definition}
\begin{note}
The name of `corner' is inspired by the appearance of these maps
when considering matrix algebras, in which case they can be arranged to
    project onto
a corner of the matrix. When $g(1)=g(q)$ we of course have $g(q^\perp)=0$,
and hence everything orthogonal to $q$ is `thrown away' by this map. 
    Thus these maps project onto a subspace \emph{where $q$ holds}.
The universal property tells us that $\{E\lvert q\}$ is the
largest such subspace.
\end{note}

\begin{definition}
	Let $q\in E$ be an effect. A \emph{filter} for $q$ is a positive subunital linear map $\xi\colon E_q\rightarrow E$ such that $\xi(1)\leq q$ and that is \emph{final} with this property: if $f\colon F\rightarrow E$ is another positive subunital linear map such that $f(1)\leq q$ then there must exist a unique $\cl{f}\colon F\rightarrow E_q$ such that $\xi\circ \cl{f} = f$. In the form of a diagram:
    \[\begin{tikzcd}[ampersand replacement = \&]
    E_q \arrow{r}{\xi}\& E \\
    F \arrow[dotted]{u}{\overline{f}}\arrow{ru}[swap]{f}\&  \\
    \end{tikzcd}\] 
\end{definition}
\begin{note}
    The name `filter' comes from the function of these maps in quantum theory as describing the act of updating the action of effects based on previous measurement outcomes, i.e.\ filtering them. The universal property can be interpreted as stating that $E_q$ is the smallest subsystem of $E$ that can faithfully represent all effects below $q$.
\end{note}

Since both these types of maps satisfy a universal property, they are (for a given effect) unique up to isomorphism. In particular, given a corner $\pi\colon E\rightarrow \{E\lvert q\}$ and an isomorphism $\Theta\colon\{E\lvert q\}\rightarrow F$ the map $\Theta\circ \pi$ is again a corner (for $q$), and furthermore any corner for $q$ is of this form. Similarly when $\xi\colon E_q\rightarrow E$ is a filter, and we have an isomorphism $\Theta\colon F\rightarrow E_q$, the map $\xi\circ \Theta$ is also a filter, and any filter for $q$ is of this form. The objects $\{E\lvert q\}$ and $E_q$ are therefore also unique up to isomorphism.
In this section we will see that there is a canonical choice of corner and filter for every effect.

As promised, a pure map is defined to be a composition of a corner and a filter:
\begin{definition}\label{def:purity}
	We call a positive subunital linear map between EJAs~$f\colon E\rightarrow F$ \emph{pure} when
    there exists some corner~$\pi$ and some filter $\xi$ (not
    necessarily for the same effect) such that $f=\xi\circ \pi$.
\end{definition}
Note that at the moment it is not yet clear whether pure maps are closed under composition or whether there are any pure maps at all.  First, we will study corners a bit more, for which we need some preparation.
\begin{proposition}\label{prop:peircedecomp}
    \cite[Proposition 1.43]{alfsen2012geometry} (Peirce-decomposition)
    Let~$E$ be an EJA with an idempotent~$p \in E$.
    Then $E_1(p):=Q_p(E):=\{Q_p(a)~;~a\in E\}$ is a sub-EJA of $E$ consisting precisely of those elements of $E$ for which $Q_p(a)=a$.
\end{proposition}

\begin{definition}
	Let $E$ be an EJA with an effect~$q \in E$.
    Then~$\{E\lvert q\} := E_1(\floor{q}) = \{E\lvert \floor{q}\}$
	and $E_q := E_1(\floor{q^\perp}^\perp) = E_1(\ceil{q})$.
\end{definition}
For an idempotent~$p$, we have $\{E\lvert p\} = E_p$.
After a few brief lemmas, we will show that~$\{E\lvert q\}$ and $E_q$
    are (the objects for) a corner and a filter respectively.

\begin{lemma}
	Let~$\omega\colon E\rightarrow \R$ be any positive linear map such that $\omega(p)=\omega(1)$ for some idempotent $p$. Then~$\omega(Q_pa)=\omega(a)$ for all $a$.
\end{lemma}
\begin{proof}
	Given such a map $\omega$ we can define $\inn{a,b}_\omega
	:= \omega(a*b)$ which is a bilinear positive semi-definite
	form. It then satisfies the Cauchy--Schwarz inequality:
	$\lvert \inn{a,b}_\omega\rvert^2 \leq \inn{a,a}_\omega
	\inn{b,b}_\omega$. Since $\omega(p)=\omega(1)$ we also have
	$\omega(p^\perp)=0$. But then $\lvert \omega(p^\perp*a)\rvert^2
	\leq \omega(p^\perp*p^\perp)\omega(a*a) = 0$ so that
	$\omega(p^\perp*a) = 0$. Then obviously $\omega(p*a) =
	\omega(a)$ from which we also get $\omega(p*(p*a))=\omega(a)$.
	Unfolding the definition of $Q_p$ we then get
	$\omega(Q_pa)=\omega(a)$.
\end{proof}
\begin{corollary}\label{cor:factorimage}
	Let~$g\colon E\rightarrow W$ be a positive linear map between EJAs
    such that $g(p)=g(1)$ for some idempotent $p$. Then $g(Q_pa)= g(a)$ for all $a$.
\end{corollary}
\begin{proof}
	Follows by the previous lemma because the states separate the maps.
\end{proof}

\begin{lemma}\label{lem:EJAfloor1}
	Let $g\colon E\rightarrow F$ be any positive linear map between EJAs
    such that $g(q)=g(1)$ for some effect $q$.
    Then~$g(\floor{q})=g(1)$.
\end{lemma}
\begin{proof}
	$g(q)=g(1)$ means that $g(q^\perp) = 0$. Write $q^\perp = \sum_i \lambda_i p_i$ where $\lambda_i>0$, then $0=g(q^\perp) = \sum_i \lambda_i g(p_i)$. Since $g$ is a positive map and $\lambda_i>0$ and $p_i\geq 0$ this implies that $g(p_i)=0$. But since $\ceil{q^\perp} = \sum_i p_i$ by proposition \ref{prop:EJAceilfloor}, we see $g(\ceil{q^\perp})=0$, so that $g(\floor{q})= g(\ceil{q^\perp}^\perp) = g(1)$.
\end{proof}

\begin{proposition}\label{eja:quot}
	Let $q$ be an effect of an EJA~$E$. Define $\pi_q\colon E\rightarrow \{E\lvert
	q\} = E_1(\floor{q})$ to be $\pi_q = r\circ Q_{\floor{q}}$
	where $r: E\rightarrow E_1(\floor{q})$ is the orthogonal
	projection map with respect to the Hilbert space structure.
    Then~$\pi_q$ is a corner for $q$. We will refer to this map
	as the \emph{standard corner} for $q$.
\end{proposition}
\begin{proof}
	First of all we have $\pi_q(1) = (r \circ Q_{\floor{q}})(1) = r (\floor{q}) = r \circ Q_{\floor{q}}(q) = \pi_q(q)$. Now suppose $g\colon E\rightarrow F$ is a positive subunital linear map such that $g(q)=g(1)$. We must show that there is a unique $\cl{g}:\{E\lvert q\} \rightarrow F$ such that $\cl{g}\circ \pi_q= g$.

	By the previous lemma $g(\floor{q})=g(1)$. Define
	$\cl{g}:E_1(\floor{q})\rightarrow F$ as the restriction of
	$g$. To prove that $\cl{g}\circ \pi_q= g$, we need
	to show that~$g(a)=g(Q_{\floor{q}}a)$ for all $a$.
    This follows from corollary \ref{cor:factorimage}.
	For uniqueness suppose we have a $h\colon E_1(\floor{q})\rightarrow
	F$ such that $h\circ \pi_q = g = \cl{g}\circ \pi_q$. Let
	$a\in E_1(\floor{q})$, then we can see $a$ as an element
	of $E$ with $\pi_q(a)=a$, so that $h(a) = h(\pi_q(a)) =
	\cl{g}(\pi_q(a)) = \cl{g}(a)$, as desired.
\end{proof}

For a positive $q=\sum_i \lambda_i p_i$ we can define a positive
square root $\sqrt{q}:=\sum_i \sqrt{\lambda_i}p_i$. This is the
unique positive element such that $\sqrt{q}*\sqrt{q}=q$.

\begin{proposition}\label{eja:filter}
	Let $q$ be an effect of an EJA~$E$. Define~$\xi_q\colon E_q \rightarrow E$ to
	be the map $\xi_q := Q_{\sqrt{q}} \circ \iota$ where $\iota$
	is the inclusion $\iota\colon E_q = E_1(\ceil{q})\rightarrow E$,
	then $\xi_q$ is a filter for $q$. We will refer to this map
	as the \emph{standard filter} for $q$.
\end{proposition}
\begin{proof}
	Clearly~$\xi_q(1)=q$. We need to show that this map is final
    with respect to this property.
    To this end, assume~$f\colon F\rightarrow E$ is any positive subunital linear map with~$f(1)\leq q$.
    We have to show that there is a unique~$\cl{f}:F\rightarrow E_q$ such that $\xi_q\circ
	\cl{f}=f$.

	Clearly $f(1)\leq q\leq \ceil{q}$.
    Thus for all $0\leq p\leq 1$ we have $f(p)\leq \ceil{q}$ so
    that $f(p)\in E_1(\ceil{q})$ which means that we can restrict
    the codomain of $f$ to $E_1(\ceil{q})=E_q$. Writing $q$ as
    $q=\sum_i \lambda_i p_i$ for some $\lambda_i>0$ and orthogonal atomic projections~$p_i$, we see it has
    a pseudo-inverse $q^{-1}:=\sum_i \lambda_i^{-1} p_i$ such that
    $Q_{\sqrt{q^{-1}}}q = q* q^{-1} = \ceil{q}$. In particular
    $Q_{\sqrt{q^{-1}}}f(p)\leq Q_{\sqrt{q^{-1}}}q = \ceil{q}$. It
    follows that the map $\cl{f}:W\rightarrow E_q$ given by
    $\cl{f}(a)=Q_{\sqrt{q^{-1}}}f(a)$ is subunital and obviously
    $(\xi_q\circ \cl{f})(a) = Q_{\sqrt{q}}Q_{\sqrt{q^{-1}}} f(a) =
    Q_{\ceil{q}}f(a) = f(a)$ by proposition \ref{prop:EJAidempotent}.

	Now for uniqueness, suppose that we have a $g:F\rightarrow E_q$ such that $\xi_q\circ g = f$. Then $Q_{\ceil{q}}\circ \iota \circ g = Q_{\sqrt{q^{-1}}} \circ Q_{\sqrt{q}} \circ \iota\circ g= Q_{\sqrt{q^{-1}}}\circ f$. As $Q_{\ceil{q}}$ acts as the identity on all elements coming from $E_1(\ceil{q})$ it can be removed from the expression. By taking the corestriction of both sides to $E_1(\ceil{q})$ we see that $g=Q_{\sqrt{q^{-1}}}\circ f = \cl{f}$, as desired.
\end{proof}

Let~$q$ be an effect in some EJA~$E$.
Note that the EJA associated to the
standard filter of $q$ is $E_q = E_1(\ceil{q})$ while the EJA of
the standard corner is $\{E\lvert q\} = E_1(\floor{q})$. 
Therefore, when $q$ is not an idempotent, we have $E_q\neq \{E\lvert q\}$and hence we cannot compose the standard filter and corner of $q$.
However, if one takes the the standard
corner of $\ceil{q}$ instead of $q$, then $E_q=\{E\lvert \ceil{q}\}$ and the filter and corner can indeed be composed.
It is easy to see that this composition~$\xi_q\circ \pi_{\ceil{q}}$ equals~$Q_{\sqrt{q}}$.
This shows that the~$Q_a$ maps are indeed pure (for positive~$a$).
Also note that the standard filter and the standard corner for the unit
    $1\in E$ are simply the identity and since a filter composed
with an isomorphism is still a filter we see that indeed all
isomorphisms are pure.
Next we will show that pure maps are closed under composition.

\subsection{The polar decomposition theorem}\label{sec:EJApurity}

The composition of two filters is again a filter,
    which can be shown in the general setting of an effectus~\cite[197IX]{basthesis}.
In our setting, it is easy to see that the composition of two corners is again a corner.
Suppose we know that a composition of a filter with a corner `in
the wrong order' can be written `in the correct order', i.e.\ that
we can always write $\pi\circ \xi$ as $\xi^\prime\circ \pi^\prime$
for some different filter $\xi^\prime$ and corner $\pi^\prime$.
Then when we have pure maps $f=\xi_1\circ\pi_1$ and $g=\xi_2\circ\pi_2$
their composition is $f\circ g = \xi_1\circ \pi_1 \circ \xi_2\circ
\pi_2$ and we can interchange $\pi_1$ and $\xi_2$ to get a composition
of two corners with two filters, which is indeed pure. So what we
need to show to establish that our definition of purity is closed
under composition is that filters and corners can be interchanged as assumed before.
It is sufficient to prove this for the standard corner and filter.
To summarize, we must show that for a given effect $q$ and idempotent~$p$
there exist effects $a$ and $b$ such that $\pi_p\circ\xi_q =
\xi_a\circ\Phi\circ \pi_b$ where $\Phi$ is some isomorphism.

The same problem of establishing that pure maps are closed under
composition in von Neumann algebras is related to the existence of
\emph{polar decompositions} of elements. By applying the polar
decomposition to $\sqrt{p}\sqrt{q}$ for positive $p$ and $q$ we
have a partial isometry $u$ such that $\sqrt{p}q\sqrt{p} =
u(\sqrt{q}p\sqrt{q})u^*$. The isomorphism $\Phi$ above is then the
conjugation map $a\mapsto uau^*$ restricted to the appropriate
domains. Partial isometries can also be defined for EJAs, and an analogous polar decomposition theorem can be stated:
\begin{definition}
    Let $\Phi\colon E\rightarrow E$ be a positive linear map on an
    EJA $E$. We denote its adjoint with respect to the inner product
    by $\Phi^*$, that is the unique linear map with~$\inn{\Phi^*(a),b} =
    \inn{a,\Phi(b)}$. We call $\Phi$ a \emph{partial isometry}
    when $\Phi\Phi^*$ and $\Phi^*\Phi$ are projections.
\end{definition}

\begin{theorem}\label{theor:polardecomp}
    Polar Decomposition: Let $p$ and $q$ be positive elements
    of a Euclidean Jordan algebra $E$. There exists a partial
    isometry $\Phi: E\rightarrow E$
    such that $Q_qQ_p = \Phi Q_{\sqrt{Q_pq^2}}$,  $\Phi(1) = \ceil{Q_q p}$, $\Phi^*(1)
    = \ceil{Q_p q}$, $\Phi^*\Phi = Q_{\ceil{Q_p q}}$
    and $\Phi\Phi^* = Q_{\ceil{Q_q p}}$.
\end{theorem}
To see how this is related to polar decomposition note that if we plug in the unit in $Q_qQ_p$ that we will get $Q_q p^2 = Q_qQ_p1 = \Phi Q_{\sqrt{Q_pq^2}}1 = \Phi(Q_p q^2)$. This polar decomposition theorem should not be confused with the already established notion of polar decomposition in Jordan algebras (see for instance \cite[Ch. VI]{faraut1994analysis}) that asserts the existence of a Jordan isomorphism between any two maximal collections of orthogonal atomic idempotents in a simple EJA. In the theory of generalized probabilistic theory, this property is also known as \emph{strong symmetry} \cite{barnum2014higher}.

The rest of this section is dedicated to proving theorem \ref{theor:polardecomp} and showing how it proves that pure maps are closed under composition.

First need a new notion:
\begin{definition}
	Let $f:E\rightarrow F$ be a positive linear map between EJAs. The
	\emph{image} of $f$ (if it exists) is the smallest effect
	$q$ such that $f(q)=f(1)$. We will denote the image of $f$
	by $\im{f}$.
\end{definition}

\begin{proposition}\label{eja:im}
	Any positive linear map $f\colon E\rightarrow F$ between Euclidean
	Jordan algebras has an image. This image is always an
	idempotent.
\end{proposition}
\begin{proof}
	Because of lemma \ref{lem:EJAfloor1} a positive linear map $f$ satisfies $f(q)=f(1)$ if and only if $f(\floor{q})=f(1)$ so we can restrict to effects which satisfy $q=\floor{q}$; viz.~the idempotents.
	
	By theorem \ref{theor:chu} EJAs are JBW-algebras (see \cite{alfsen2012geometry}) so that the idempotents form a complete lattice. Furthermore, all states are normal, meaning they preserve infima. Because the states separate the maps, all maps are also normal. We conclude that $\im{f} = \inf\{p~;~p^2=p, f(p)=f(1)\}$ exists and that $f(\im{f}) = f(\inf\{p~;~f(p)=f(1)\})=\inf_p f(p) = f(1)$.
\end{proof}

\begin{proof}[Proof of Theorem \ref{theor:polardecomp}]
	Let $\Phi = Q_qQ_pQ_{(Q_p q^2)^{-1/2}}$ so that $\Phi^* = Q_{(Q_p q^2)^{-1/2}}Q_pQ_q$because $Q_a^*=Q_a$ for all $a$. Then $Q_{(Q_p q^2)^{1/2}}\Phi^* = Q_{\ceil{Q_pq^2}}Q_pQ_q = Q_p Q_q$. By taking adjoints we then get $Q_qQ_p = \Phi Q_{\sqrt{Q_p q^2}}$ as desired. Note that 
	\begin{equation*}
	\Phi^*\Phi  \ =\  Q_{(Q_p q^2)^{-1/2}}Q_pQ_qQ_qQ_pQ_{(Q_p
	q^2)^{-1/2}} \ =\  Q_{(Q_p q^2)^{-1/2}} Q_{Q_p q^2} Q_{(Q_p
	q^2)^{-1/2}} \ =\  Q_{\ceil{Q_p q^2}}
	\end{equation*}
	by application of the fundamental equality. Since $\ceil{Q_p q^2} = \ceil{Q_p \ceil{q^2}} = \ceil{Q_p \ceil{q}} = \ceil{Q_p q}$ this can be simplified to $\Phi^*\Phi = Q_{\ceil{Q_p q}}$.  Because $\Phi^*\Phi$ is a projection we can use \cite[Proposition 6.1.1]{kadison2015fundamentals} to conclude that $\Phi\Phi^*$ must be projection as well. By a simple calculation $\Phi^*(1) = \ceil{Q_p q}$ so that it remains to show that $\Phi(1) = \ceil{Q_q p}$ and that $\Phi\Phi^*(1)=\ceil{Q_q p}$ since this latter condition (in combination with the knowledge that $\Phi\Phi^*$ is a projection) is sufficient to conclude that $\Phi\Phi^* = Q_{\ceil{Q_q p}}$.

	Suppose $\Phi^*(s)=0$ then $\inn{1,\Phi^*(s)} = 0 = \inn{Q_pQ_q s, (Q_p q^2)^{-1/2}}$. Since $\ceil{(Q_p q^2)^{-1/2}} = \ceil{Q_p q^2}$ this gives $0=\inn{Q_pQ_q s, Q_p q^2} = \inn{s, Q_qQ_{p^2}Q_q 1} = \inn{s, Q_{Q_q p^2}1} = \inn{s, Q_q p^2}$. We conclude that $\Phi^*(s)=0$ if and only if $s\perp \ceil{Q_q p}$ so that $\im{\Phi^*} = \ceil{Q_q p}$. We of course also have $0=\inn{1,\Phi^*(s)} = \inn{\Phi(1),s}$ so that $\ceil{\Phi(1)} = \im{\Phi^*}= \ceil{Q_q p}$. Because $\inn{1,(\Phi(1))^2} = \inn{\Phi(1),\Phi(1)} = \inn{\Phi^*\Phi(1),1} = \inn{\ceil{Q_p q}, 1} = \inn{\Phi^*(1),1} = \inn{1,\Phi(1)}$ we conclude that $\Phi(1)=(\Phi(1))^2$ so that $\Phi(1)=\ceil{\Phi(1)} = \ceil{Q_q p}$. 

	By a similar argument as above we can show that
	$\im{\Phi\Phi^*}=\ceil{(\Phi\Phi^*)(1)}$ which gives
	$(\Phi\Phi^*)(1) \leq \im{\Phi\Phi^*} \leq \im{\Phi^*} =
	\ceil{Q_q p}$. For the other direction we recall that we
	had $Q_q Q_p = \Phi Q_{\sqrt{Q_pq^2}}$ so that $Q_{Q_qp^2}
	= Q_q Q_p Q_p Q_q = \Phi Q_{\sqrt{Q_pq^2}}Q_{\sqrt{Q_pq^2}}\Phi^*
	= \Phi Q_{Q_pq^2}\Phi^* \leq \| Q_{Q_pq^2} \|\Phi\Phi^*$.
	By inserting the unit into the expression and taking the
	ceiling we are left with $\ceil{Q_qp} = \ceil{Q_{Q_qp^2}1}
	\leq \ceil{(\Phi\Phi^*)(1)}$.
\end{proof}

\begin{proposition}
	Let $\xi_q: E_1(\ceil{q})\rightarrow E$, $\xi_q =
	Q_{\sqrt{q}}\colon \iota$ be the standard filter of an
	effect $q$ and $\pi_p\colon E\rightarrow E_1(p)$, $\pi_p = r\circ
	Q_p$ be the standard corner of an idempotent effect $p$.
	Then $\pi_p \circ \xi_q = \xi_a\circ \Phi\circ \pi_b$ where
	$a$ and $b$ are some effects and $\Phi$ is an isomorphism.
	In other words: $\pi_p\circ \xi_q$ is pure.
\end{proposition}
\begin{proof}
	Define the shorthand $q\&p:= Q_{\sqrt{q}}(p)$. Let $f = \pi_p \circ \xi_q: E_{\ceil{q}} \rightarrow E_p$. Because $f(1) = \pi_p (\xi_q(1)) = \pi_p(q) = p\& q$ we see that there must exist $\cl{f}: E_{\ceil{q}}\rightarrow E_{\ceil{p\&q}}$ such that $\xi_{p\& q}\circ \cl{f} = f$ where $\xi_{p\& q}:E_{\ceil{p\&q}}\rightarrow E_p$ by the universal property of the filter. This $\cl{f}$ is given by $\cl{f} = Q_{(p\& q)^{-1/2}}\circ f$ so that $\cl{f}(1) = Q_{(p\& q)^{-1/2}}(p\& q) = \ceil{p\& q} = 1$ since the codomain is $E_{\ceil{p\& q}}$. We will ignore the restriction and inclusion maps present in the filter and corner so that we can write $f = Q_p Q_{\sqrt{q}}$ and similarly $\cl{f} = Q_{(p\& q)^{-1/2}}Q_p Q_{\sqrt{q}}$.

	Similar to the argument used in the proof of theorem \ref{theor:polardecomp} we can show that $\im{\cl{f}} = \ceil{q\& p}$. Then we can use the universal property of the corner to find a map $\Phi: E_{\ceil{q\& p}}\rightarrow E_{\ceil{p\& q}}$ such that $\Theta\circ \pi_{\ceil{q\& p}} = \cl{f}$. Because $\cl{f}$ and $\pi_{\ceil{q\& p}}$ are unital, $\Phi$ has to be unital as well. Note that $\Phi$ is just a restriction of $\cl{f}$ to the appropriate domain and that $\cl{f} = Q_{(p\& q)^{-1/2}}Q_p Q_{\sqrt{q}}$ is exactly the same as $\Phi^*$ in the proof of theorem \ref{theor:polardecomp}. We can conclude as a consequence that $\Phi\Phi^* = Q_{\ceil{p\& q}}$ while $\Phi^*\Phi = Q_{\ceil{q\& p}}$. These are of course the identity maps on $E_{\ceil{p\&q}}$ respectively $E_{\ceil{q\& p}}$ so that $\Phi^* = \Phi^{-1}$. We conclude that $f = \xi_{p\& q} \circ \cl{f} = \xi_{p\& q}\circ \Phi \circ \pi_{\ceil{q\& p}}$ where $\Phi$ is an isomorphism.
\end{proof}
\begin{corollary}\label{cor:purepure}
The composition of pure maps is pure.
\end{corollary}
\begin{proof}
	Let $f_1$ and $f_2$ be pure, then $f_i=\xi_i\circ \Theta_i\circ \pi_i$, so that $f_1\circ f_2 = \xi_1\circ \Theta_1\circ \pi_1 \circ \xi_2 \circ \Theta_2\circ \pi_2 = \xi_1^\prime \circ \xi^\prime\circ \Theta^\prime \circ \pi^\prime \circ \pi_2^\prime$ by the previous proposition and writing $\xi_1\circ \Theta_1 = \xi_1^\prime$ and $\Theta_2\circ \pi_2 = \pi_2^\prime$ where $\xi_1^\prime$ and $\pi_2^\prime$ are again a filter respectively a corner. But now since a composition of filters is again a filter and a composition of corners is again a corner we see that $f_1\circ f_2$ is indeed pure.
\end{proof}

\section{Diamond adjointness and positivity}\label{sec:diamond}
Since Euclidean Jordan algebras are also Hilbert spaces, we can find for any positive map an adjoint with respect to the inner product. This means that the category of all EJAs with positive (not necessarily subunital) maps is a dagger category. The adjoint of a subunital map is not necessarily subunital again however, so that $\EJA_{psu}$ is \emph{not} a dagger category. However, the set of pure maps is closed under taking adjoints (which can be shown by a simple case analysis), so that this restricted category \emph{is} a dagger category. This is not an accident: a consequence of the results in this section will be that $\EJA_{psu}$ is a \emph{$\dagger$-effectus}, a type of structure already introduced as an abstract version of the category of von Neumann algebras where the pure maps also form a dagger category. For more information regarding $\dagger$-effectuses we refer to \cite{basthesis}.

An important notion in an effectus is that of $\diamond$-adjointness. This is a possibilistic alternative to adjointness that can be defined even when there is no obvious choice of dagger. In this section we will study $\diamond$-adjointness in $\EJA_{psu}$ and show that it behaves similarly to $\diamond$-adjointness in von Neumann algebras. In particular we will give a characterization of pure $\diamond$-self-adjoint maps and show that a pure $\diamond$-\emph{positive} map $f:E\rightarrow E$ is completely determined by its image at the unit: $f = Q_{\sqrt{f(1)}}$ and thus that the only pure $\diamond$-positive maps are the quadratic representation maps $Q_a$ for some positive $a$. As these quadratic representation maps are the Jordan equivalent of the \emph{sequential product} map $b\mapsto aba$ \cite{wetering2018characterisation}, this can be seen as a new characterization of the sequential product.

\begin{definition}
Let $f: E\rightarrow F$ be a positive subunital map and write $\Idem(E)$
for the set of idempotents of~$E$. Define
the maps $f^\diamond\colon \Idem(F)\to\Idem(E)$
and $f_\diamond\colon \Idem(E)\to\Idem(F)$
by 
\begin{equation*}
f^\diamond(p)\ =\ \ceil{p\circ f}
	\qquad\text{and}\qquad
f_\diamond(q)\ =\ \im( Q_q\circ f ).
\end{equation*}
We say that $f:E\rightarrow F$ is \emph{$\diamond$-adjoint} to $g:F\rightarrow E$ when $f^\diamond = g_\diamond$ or equivalently $f_\diamond = g^\diamond$ \cite{basthesis}. We call $f:E\rightarrow E$ \emph{$\diamond$-self-adjoint} when $f$ is $\diamond$-adjoint to itself, and we call $f$ \emph{$\diamond$-positive} when there exists a $\diamond$-self-adjoint $g$ such that $f=g\circ g$.
\end{definition}

It can be shown that $f^\diamond(p)\leq q^\perp$ iff $f_\diamond(q)\leq p^\perp$ so that the diamond defines a Galois connection between the orthomodular lattices of idempotents. As a result we get a functor $\diamond: \EJA_{psu}\rightarrow $ \textbf{OMLatGal} \cite{basthesis}.
Note that $\diamond$-self-adjointness is weaker than regular self-adjointness:
\begin{proposition}\label{prop:diamond-adjointness}
    Any self-adjoint operator~$f:E\rightarrow E$ on an EJA~$E$
        is~$\diamond$-self-adjoint.
    In particular $Q_a$ is $\diamond$-self-adjoint for any~$a \in E$.
    Consequently, $Q_a$ is $\diamond$-positive for positive~$a$.
\end{proposition}
\begin{proof}
Let~$f$ be any self-adjoint operator.
It suffices to show~$f^\diamond(s) \leq t^\perp \iff f^\diamond(t) \leq s^\perp$
        for all idempotents~$s,t \in E$
    (see~\cite[\S207III]{basthesis}).
This is equivalent to
    \begin{equation}\label{selfadjointform}
        \inn{f^\diamond(s),t}\ =\ 0\quad \iff\quad  \inn{s,f^\diamond(t)} \ =\  0 \qquad (s,t \in E \text{ idempotents}).
    \end{equation}
By the spectral theorem~$\inn{\ceil{q},s} = 0 \iff \inn{q,s}=0$
    for any positive~$q$ and idempotent~$s$,
    so \eqref{selfadjointform}
    is equivalent to~$\inn{f(s),t} = 0 \iff \inn{s,f(t)}=0$,
        which clearly holds as~$f$ is self-adjoint.

Pick any positive~$a\in E$.
    By the fundamental identity, we have~$Q_a = Q_{\sqrt{a}^2} = Q_{\sqrt{a}}^2$,
        so~$Q_a$ is the square of a~$\diamond$-self-adjoint map,
            hence~$\diamond$-positive.
\end{proof}

The rest of this section contains the necessary work to prove the following theorem characterizing the pure $\diamond$-positive maps:
\begin{theorem}
    \label{super-duper-theorem}
    Let $g: E\rightarrow E$ be a pure $\diamond$-positive map and write $p:= g(1)$, then $g = Q_{\sqrt{p}}$.
\end{theorem}

First, we will need a well-known fact about Jordan algebras for which we need a short lemma.

\begin{lemma}\label{ordersharpprop}
    An effect $p$ is called \emph{order-sharp} when $q=0$ whenever both $q\leq p$ and $q\leq p^\perp$. An effect $p$ is order-sharp if and only if it is an idempotent.
\end{lemma}
\begin{proof}
    Let $a$ be an order-sharp effect and write $a=\sum_i \lambda_i p_i$. Let $r_i=\min\{\lambda_i,1-\lambda_i\}$, then $r_ip_i\leq a$ and $r_ip_i\leq a^\perp=1-a$ which implies that $r=0$, so either $\lambda_i=1$ or $\lambda_i=0$ for all $i$. But then as $a$ is a sum of orthogonal idempotents it is also an idempotent. For the other direction suppose $a$ is idempotent. Let $q\leq a$ and $q\leq a^\perp$. By $q\leq a$ we know that $Q_a q = q$, but we also have $Q_a q \leq Q_a a^\perp = 0$ so that $q=0$.
\end{proof}
\begin{proposition}
	A unital order isomorphism between EJAs is a Jordan isomorphism: that is, it preserves the Jordan multiplication.
\end{proposition}
\begin{proof}
Let~$\Theta\colon E \to F$ be any unital order-isomorphism
    between EJAs.
    As~$2(a * b) = (a+b)^2 - a^2 - b^2$
        it suffices to show~$\Theta(a)^2= \Theta(a^2)$ for any~$a \in E$.
Write~$\sum_i \lambda_i p_i = a$ for the spectral decomposition of~$a$.
As idempotents are exactly the order-sharp elements by the previous lemma
    and idempotents~$p,q$ are orthogonal iff~$p \leq 1-q$,
    we see that~$\Theta(p_i)$ are also pairwise orthogonal idempotents.
    Thus~$\Theta(a)^2 = (\sum_i \lambda_i \Theta(p_i))^2 = \Theta(\sum_i \lambda_i^2 p_i) = \Theta(a^2)$,
        as desired.
\end{proof}
\begin{corollary}
	Let $\Theta:E\rightarrow F$ be any unital order-isomorphism between Euclidean
        Jordan Algebras.
    Then, for any~$a,b\in E$, we have~$\Theta(Q_ab) = Q_{\Theta(a)}\Theta(b)$.
    That is: $\Theta\circ Q_a = Q_{\Theta(a)}\circ \Theta$.
    Equivalently: $Q_a\circ \Theta = \Theta \circ Q_{\Theta^{-1}(a)}$.
\end{corollary}

The next few results involve the notion of \emph{faithfulness}. A
map $f:E\rightarrow F$ is called faithful when for any positive $a$
the equation $f(a)=0$ implies $a=0$. A map $f$ is faithful if and
only if $\im{f} = 1$.
\begin{lemma}
	Let $f:E\rightarrow E$ be a faithful pure $\diamond$-self-adjoint map between
    EJAs. Then $f=Q_{\sqrt{f(1)}}\circ \Theta$ for some unital Jordan isomorphism~$\Theta$
    with~$\Theta(\sqrt{f(1)})=\sqrt{f(1)}$ and $\Theta=\Theta^{-1}$.
\end{lemma}
\begin{proof}
    First, we collect some basic facts.
	As~$f$ is pure, we have~$f = \xi\circ \pi$ for
	some filter~$\xi$ and corner~$\pi$.
    Note
    that~$\im{f}=\im{\pi}$ as~$\xi$ is faithful and~$f(1)=\xi(1)$
	as $\pi$ is unital.
    Hence, by~$\diamond$-self-adjointness of~$f$,
	we have~$\ceil{\xi(1)} = \ceil{f(1)} = f^\diamond(1)=f_\diamond(1)
	    = \im{f} = \im{\pi}$.
        Next, by the universal properties
        of filters and corners,
        there exist order
        isomorphisms~$\Theta_1,\Theta_2$ such that~$\xi = \xi_{\xi(1)}\circ \Theta_1$ and
	$\pi = \Theta_2\circ\pi_{\im{\pi}}$, so that $f=\xi_{\xi(1)}\circ
	\Theta_1\circ \Theta_2 \circ \pi_{\im{\pi}} = \xi_{f(1)}\circ
	\Theta \circ \pi_{\im{f}}$ defining $\Theta :=
	\Theta_1\circ\Theta_2$.

    We assumed~$f$ is faithful, i.e.~$\im f = 1$.
    So~$\pi_{\im f} = \pi_1=\id$.
    For brevity, write~$q := \sqrt{f(1)}$.
    As~$\ceil{q}=\ceil{f(1)}=\ceil{\xi(1)} = \im \pi = 1$,
        we have~$\xi_q = Q_q$
        and so~$f = Q_q\after \Theta$.

	As seen in the proof of proposition~\ref{prop:diamond-adjointness} 
    if $f$ is
	$\diamond$-self-adjoint we have $\inn{f(a),b}=0 \iff
	\inn{a,f(b)}=0$. In this case this translates to
	$0=\inn{Q_q\Theta(a),b}=0 \iff 0=\inn{a,Q_q\Theta(b)} =
	\inn{Q_qa,\Theta(b)} = \inn{\Theta^{-1}Q_qa,b}$. This implies
	that $\ceil{Q_q\Theta(a)} = \ceil{\Theta^{-1}Q_qa} =
	\Theta^{-1}(\ceil{Q_qa})$ for all $a$. Write $q=\sum_i\lambda_i
	q_i$ where the $q_i$ are atomic. Then we have $Q_qq_i =
	\lambda_i^2 q_i$. Filling in $a=q_i$ we then get
	$\ceil{Q_q\Theta(q_i)} = \Theta^{-1}(\ceil{Q_qq_i}) =
	\Theta^{-1}(\ceil{\lambda_i^2 q_i}) = \Theta^{-1}(q_i)$.
	The right-hand side is atomic as Jordan isomorphisms preserve
	atomicity, so the left-hand side must also be atomic. Since
	$\ceil{b}$ is atomic if and only if $b$ is proportional to
	an atomic predicate we then get $Q_q\Theta(q_i) = \mu_i
	\Theta^{-1}(q_i)$ for some $0<\mu_i<1$. By composing with $\Theta$ this becomes $(\Theta Q_q \Theta)(q_i) = \mu_i q_i$. Now we note that:
    \begin{equation*}
	\sum_i\lambda_i^2 \Theta(q_i) \ = \  \Theta(q^2) \ = \
	\Theta Q_q \Theta(1) \ = \  \sum_i \Theta Q_q \Theta (q_i)
	\ = \  \sum_i \mu_i q_i
    \end{equation*}
	Now let $p_j = \sum_{i, \lambda_i=\lambda_j} q_i$ and $r_j = \sum_{i, \mu_i=\mu_j} q_j$. 
    Then we can write $\sum_i \lambda_i^2 \Theta(q_i) = \sum_j \lambda_j^2 \Theta(p_j)$ and $\sum_i \mu_i q_i = \sum_j \mu_j r_j$ 
    where in the sums on the right-hand side each of the $\lambda_j$ and each of the $\mu_j$ is distinct. 
    Since $\Theta$ preserves orthogonality this means we get two orthogonal decompositions that are equal: $\sum_j \lambda_j^2 \Theta(p_j) = \sum_j \mu_j r_j$. 
    By uniqueness of such decompositions we then have $\lambda_j^2 =\mu_j$ and $\Theta(p_j) = r_j$ (where we assume for now that we have ordered the eigenvalues from high to low). 
    But of course since the $\lambda_j^2$ and $\mu_j$ agree, the $p_j$ and the $r_j$ will also agree by their definition, so that $\Theta(p_j) = p_j$. Finally, we get $\Theta(q) = \sum_j \lambda_j \Theta(p_j) = \sum_j\lambda_j p_j = q$.

	Now $\Theta Q_q =  Q_{\Theta(q)}\Theta = Q_q\Theta$ so the
	$\Theta$ commutes with $Q_q$. Note that since $\ceil{q}=1$,
	$q$ will be invertible. Let $g = Q_q$. Then $g^\diamond
	\Theta^\diamond = (\Theta g)^\diamond = (g\Theta)^\diamond
	= f^\diamond = f_\diamond = (g\Theta)_\diamond = g_\diamond
	\Theta_\diamond = g^\diamond (\Theta^{-1})^\diamond$. Now
	$g^{-1}$ is not a subunital map, but it can be scaled
	downwards until it is, in which case $gg^{-1} = \lambda
	\id$ for some $\lambda>0$, in which case $g^\diamond
	(g^{-1})^\diamond = \id$. Since $g^\diamond$ has an inverse,
	we see that $g^\diamond \Theta^\diamond = g^\diamond
	(\Theta^{-1})^\diamond$ can only hold when $\Theta^\diamond
    = (\Theta^{-1})^\diamond$. From this and $\Theta(\ceil{a}) = \ceil{\Theta(a)}$
        it follows that~$\Theta=\Theta^{-1}$.
\end{proof}

\begin{proposition}
	Let $g\colon E\rightarrow E$ be a faithful pure $\diamond$-positive map and let $p:=g(1)$, then $g = Q_{\sqrt{p}}$.
\end{proposition}
\begin{proof}
	Since $g$ is $\diamond$-positive there must exist some pure $\diamond$-self-adjoint $f:E\rightarrow E$ such that $g = ff$. Since $g$ is faithful, the $f$ must be faithful as well since $1=\im{g} = \im{ff} \leq \im{f}$. By the previous lemma $f=Q_q\Theta$ for some $q$ and $\Theta(q) = q$ and $\Theta=\Theta^{-1}$. But then $g = ff = Q_q\Theta Q_q \Theta = Q_qQ_{\Theta(q)}\Theta\Theta = Q_qQ_q\Theta^{-1}\Theta = Q_{q^2}$. Now $g(1) = Q_{q^2}1 = q^4$, so that $g= Q_{\sqrt{p}}$ where $p=q^4$.
\end{proof}

\begin{proof}[Proof of theorem~\ref{super-duper-theorem}]
	Let $f$ be $\diamond$-self-adjoint so that in particular $\im{f} = \ceil{f(1)}$. We can then corestrict $f$ to $E\rightarrow E_1(\im{f})$. Using corollary \ref{cor:factorimage} we also get $f=fQ_{\im{f}}$ so that we can factor $f$ as $f = \xi_{\im{f}}\circ \cl{f}\circ \pi_{\im{f}}$ where $\cl{f}:E_1(\im{f})\rightarrow E_1(\im{f})$. Note that $\xi_{\im{f}}$ is nothing but the inclusion map into $E$. It is also easy to see that $\cl{f} = \pi_{\im{f}}\circ f \circ \xi_{\im{f}}$. Then $\cl{f}(1) = \pi_{\im{f}}(f(\im{f})) = f(\im{f})=f(1)$. Because $f$ is pure, $\cl{f}$ is also pure as it is a composition of pure maps. When $f$ is $\diamond$-self-adjoint we have $\cl{f}^\diamond = (\pi_{\im{f}}\circ f\circ \xi_{\im{f}})^\diamond = \xi_{\im{f}}^\diamond \circ f^\diamond \circ \pi_{\im{f}}^\diamond = (\pi_{\im{f}})_\diamond \circ f_\diamond \circ (\xi_{\im{f}})_\diamond = \cl{f}_\diamond$. Here we have used that $\xi_s^\diamond(t) = \ceil{\xi_s(t)} = \ceil{t} = t$ since $t\in E_1(s)$ and $(\pi_s)_\diamond(t) = \im{\pi_t\circ \pi_s} = \im{\pi_t} = t$ which also follows because $t\leq s$, so that $\xi_s^\diamond = (\pi_s)_\diamond$.

	When $f$ is $\diamond$-self-adjoint we get $\im{f^2} = f^\diamond(f_\diamond(1)) = f^\diamond(\im{f}) = f^\diamond(1) = f_\diamond(1)=\im{f}$. For a $\diamond$-self-adjoint $f$ we then get $\cl{f}^2 = \pi_{\im{f}}\circ f \circ \xi_{\im{f}}\circ \pi_{\im{f}}\circ f \circ \xi_{\im{f}} = \pi_{\im{f}}\circ f^2 \circ \xi_{\im{f}} = \cl{f^2}$. We conclude that when $g=f\circ f$ is $\diamond$-positive, $\cl{g}=\cl{f^2}=\cl{f}^2$ is also $\diamond$-positive. Since $\cl{g}$ is faithful we have already established that $\cl{g} = Q_{\sqrt{p}}$ where $p = \cl{g}(1)=g(1)$ and $\im{g} = \ceil{g(1)} = \ceil{p}$. Now $g=\xi_{\im{g}} Q_{\sqrt{p}} \pi_{\im{g}} = Q_{\sqrt{p}}Q_{\ceil{p}} = Q_{\sqrt{p}}$.
\end{proof}


We're now ready to conclude that
    $\EJA_{psu}^{\opp}$
    is a~$\dagger$-effectus.
The details of $\dagger$-effectuses are beyond the scope we chose for this paper;
for those we refer the reader to~\cite[\S173--215]{basthesis}.
\begin{theorem}
    $\EJA_{psu}^{\opp}$ is a $\dagger$-effectus (as defined in \cite[\S215]{basthesis}).
\end{theorem}
\begin{proof}
It is straightforward to show that~$\EJA_{psu}^\opp$
    is an effectus, see eg.~\cite[\S191]{basthesis}.
    The existence of corners (prop.~\ref{eja:quot}),
    filters (prop.~\ref{eja:filter}), images~(prop.~\ref{eja:im})
    and the fact that the complement of an idempotent is again idempotent,
    shows that~$\EJA_{psu}^\opp$ is a~$\diamond$-effectus.~\cite[\S206]{basthesis}
This, combined with  the characterization of~$\diamond$-positive maps (thm.~\ref{super-duper-theorem})
    and the fact that pure maps are closed under composition (cor.~\ref{cor:purepure}),
    gives us that~$\EJA_{psu}^\opp$ is a~$\&$-effectus.~\cite[\S211]{basthesis}
To see, finally, that we have a~$\dagger$-effectus at hand,
    it is sufficient to show the three conditions from \cite[\S215III]{basthesis}.
\begin{enumerate}
\item
    Every predicate~$0 \leq p\leq 1$ must have a unique square root:
	 there must exist a unique $0 \leq q\leq 1$
	such that $q\& q = p$. This is true by the spectral theorem
        (corollary~\ref{cor:spectral}).
\item
    We must have
	$Q(\sqrt{Q_{\sqrt{p}}q})^2 = Q_{\sqrt{p}}Q(\sqrt{q})^2
	Q_{\sqrt{p}}$.
        Rewritten to $Q(Q_{\sqrt{p}}q) =
	Q_{\sqrt{p}}Q_qQ_{\sqrt{p}}$, this becomes the familiar
        fundamental equality for the quadratic representation.
\item
	The filter of an idempotent must map idempotents to idempotents.
        This is clearly true for the standard filter
            of an idempotent (being an inclusion) and
            hence for any as isomorphisms preserve idempotents as well. \qedhere
\end{enumerate}
\end{proof}

\section{Conclusion}
We have shown that the definition of purity for von Neumann algebras~\cite{bramthesis,westerbaan2016paschke,westerbaan2016universal} and effectus theory~\cite{cho2015introduction,cho2015quotient} also works well in the category of positive contractive linear maps between Euclidean Jordan algebras. In particular, this definition of purity is closed under composition and the dagger, and includes the equivalents of the L\"uders maps for Jordan algebras. We have also shown that the possibilistic notion of $\diamond$-adjointness from effectus theory translates to these algebras, and we have used it to give a new characterization of the `sequential product' maps $b\mapsto Q_ab$ for positive $a$.

With the results in this paper we know that the filter-corner definition of purity works well in von~Neumann algebras and in Euclidean Jordan algebras. 
This begs the question whether our results concerning pure maps generalize
to the class of~\emph{JBW-algebras} \cite{hanche1984jordan},
which include all von~Neumann algebras and EJAs.
We consider this to be a challenging topic for future research.

The precise relationship between the different notions of purity found in the literature and the one studied in this paper is yet to be determined. When restricted to complex matrix algebras all the definitions coincide, but when direct sums of simple algebras are considered the definitions sometimes diverge. For example, the identity map on a direct sum is not atomic and therefore not pure in the sense of \cite{chiribella2011informational}. Also, the adjoint of a pure map in the sense of \cite{cunningham2017purity,selby2018reconstructing} on a direct sum need not be pure.

Then finally there is the matter of the different notions of positivity for
maps between Euclidean Jordan algebras: they can be
\emph{superoperator positive} (mapping positive elements to positive elements),
\emph{operator positive} as linear maps between Hilbert spaces,
and, of course, $\diamond$-positive. 
Any relation? 
While preliminary investigations reveal
none between superoperator positivity 
and operator positivity, 
$\diamond$-positivity does seem to be connected with operator positivity.

\textbf{Acknowledgments}: This work is supported by the ERC under
the European Union's Seventh Framework Programme (FP7/2007-2013) /
ERC grant n$^\text{o}$ 320571.

\bibliographystyle{eptcs}
\bibliography{main}

\appendix

\section{Basic structure of EJAs}
EJAs are commonly defined to be finite-dimensional. 
The infinite-dimensional algebras we study are also known as JH-algebras \cite{chu2017infinite,chu2011jordan}
(where we additionally require the existence of a unit). 
In this appendix we will give a relatively self-contained proof that the EJAs we use are JBW-algebras and that every element has a finite spectral decomposition.

\begin{proposition}
For every EJA~$E$
there is a constant~$r>0$ such that, for all~$a,b\in E$,
\begin{equation}
	\label{prod-bounded}
	\|\,a*b\,\|_2\ \leq\  r\,\|a\|_2\,\|b\|_2,
\end{equation}
where~$\|c\|_2 \equiv \sqrt{\left<c,c\right>}$
denotes the Hilbert norm.
In particular, $*$ is uniformly continuous with respect to the
Hilbert norm.
\end{proposition}
\begin{proof}
The trick is to apply
the \emph{uniform boundedness principle} (see e.g.~\cite{sokal2011})
twice,
which states that any collection~$\mathcal{T}$
of bounded operators
from a Banach space~$X$ to a normed vector space~$Y$
that is bounded pointwise,
i.e.~$\sup_{T\in\mathcal{T}} \|Tx\|<\infty$
for all~$x\in X$,
is uniformly bounded
in the sense that $\sup_{T\in \mathcal{T}} \|T\|<\infty$.

For the moment fix~$a\in E$.
Our first step is to show that the operator
	$a*(\,\cdot\,)\colon E\to E$ is bounded.
	To this end
 consider
the collection of linear functionals $\left<b,a*(\,\cdot\,)\right>
\colon E\to\R$,
where $b\in E$ with $\|b\|_2\leq 1$.
These are bounded operators,
since~$\left<b,a*(\,\cdot\,)\right>
= \left<a*b,(\,\cdot\,)\right>$,
and as a collection they 
are bounded pointwise,
since
$\left|\left<b,a*c\right>\right| \leq \|b\|_2\, \|a*c\|_2
\leq\|a*c\|_2<\infty$.
Hence 
\begin{equation*}
r_a\ :=\ \sup_{\|b\|_2\leq 1}\|
\left<a*b,(\,\cdot\,)\right>\|\ <\ \infty
\end{equation*}
by the uniform boundedness principle.
Since in particular
	$(\|a*b\|_2)^2 = \left<a*b,a*b\right>
\leq r_b\|a*b\|_2$ for all~$b\in E$ with~$\|b\|_2\leq 1$,
we get
$\|a*b\|_2\leq r_a$ for all~$b\in E$
with $\|b\|_2\leq 1$,
and thus~$\|a*b\|_2\leq r_a\|b\|_2$
for any~$b\in E$.
In other words,
the linear operator $a*(\,\cdot\,)\colon E\to E$
is bounded.

Now, to prove equation~\eqref{prod-bounded}
it suffices to show that $\sup_{\|a\|_2\leq 1} \|a*(\,\cdot\,)\|$ is finite.
For this, in turn,
it suffices,
by the uniform boundedness principle,
to show given~$b\in E$ that $\sup_{\|a\|_2\leq 1}\|a*b\|_2 <
	\infty$.
	Since~$\|a*b\|_2=\|b*a\|_2\leq \|b*(\,\cdot\,)\|\,
	\|a\|_2
	\leq \|b*(\,\cdot\,)\|<\infty$
	for all~$a\in E$ with $\|a\|_2\leq 1$,
	this is indeed the case.
\end{proof}
To proceed we need some basic algebraic properties
of Jordan algebras, which are most conveniently expressed
with some additional notation.
\begin{notation}
Let~$E$ be a Jordan algebra.
\begin{enumerate}
\item
We write $a^0:=1$,\quad $a^1:= a$,\quad  $a^2:= a*a$,\quad  
		$a^3:= a*a^2$,\quad 
$a^4:=a*a^3$, \dots .
Note that since~$*$ is not associative
it's not a priori clear whether equations like  $a^4 = a^2 * a^2$ hold.
\item
Given~$a\in E$ we denote the linear operator $E\to E\colon b\mapsto a*b$
by~$L_a$.

Given two linear operators $S,T\colon E\to E$
we write  $[S,T]:= ST-TS$ for the commutator of~$S$ and~$T$.
\end{enumerate}

\end{notation}
\begin{proposition}
\label{jaeqs}
Given a Jordan algebra~$E$,
and $a,b,c\in E$, we have
\begin{enumerate}
\item
\label{jaeqs-1}
$[L_a,L_{a^2}] = 0$;\quad
$[L_b,L_{a^2}] = 2[L_{a*b},L_{a}]$;\quad and\quad  
$[L_a,L_{b*c}] + [L_b,L_{c*a}] + [L_c,L_{a*b}] = 0$;
\item
\label{jaeqs-2}
$L_{a*(b*c)} = L_aL_{b*c} + L_b L_{c*a} + L_c L_{a*b} - L_b L_a L_c - 
L_c L_a L_b$;
\item
\label{jaeqs-3}
$a^n*(b*a^m)=(a^n*b)*a^m$
\quad and\quad
		$a^n * a^m = a^{n+m}$\quad
	for all $n,m\in \N$.
\end{enumerate}
\end{proposition}
\begin{proof}
\ref{jaeqs-1}.\ 
The first equation, $[L_a,L_{a^2}]=0$,
is just a reformulation
of the Jordan identity:
\begin{equation*}
		L_aL_{a^2}b\,\equiv\, a*(b*a^2) \ =\  (a*b)*a^2
	\,\equiv\, L_{a^2}L_a b.
\end{equation*}
Note that $[L_{a+b},L_{(a+b)^2}]-
[L_{a-b},L_{(a-b)^2}] = 
4[L_a,L_{a*b}] + 2[L_b,L_{a^2}] + 2[L_b,L_{b^2}]$---just expand both sides.
Applying
$[L_d,L_{d^2}]=0$ with $d=b, \,a+b,\, a-b$,
we get  $[L_b,L_{a^2}]
	= -2[L_{a},L_{a*b}]=2[L_{a*b},L_a]$.
Similarly, one gets
$[L_a,L_{b*c}] + [L_b,L_{c*a}] + [L_c,L_{a*b}] = 0$
by expanding
	$2[L_{(a+c)*b},L_{a+c}]-[L_b,L_{(a+c)^2}]\equiv 0$.

\ref{jaeqs-2}.\ 
Since $[L_a,L_{b*c}] + [L_b,L_{c*a}] + [L_c,L_{a*b}] = 0$,
we have, for all~$d\in E$,
\begin{equation*}
	(L_a L_{b*c}  +
	L_b L_{c*a} +
	L_c L_{a*b})d \ = \ 
(b*c)*(a*d)\,+\,
(c*a)*(b*d)\,+\,
(a*b)*(c*d).
\end{equation*}
Since the right-hand side 
of this equation is invariant under a switch of the roles 
of~$a$ and~$d$,
so must be the left-hand side,
which gives us the professed equality after some rewriting:
\begin{alignat*}{3}
(L_a L_{b*c}  +
L_b L_{c*a} +
L_c L_{a*b})d \ &= \ 
(L_d L_{b*c}  +
L_b L_{c*d} +
	L_c L_{d*b})a\qquad&& a \leftrightarrow d\\
\ &\equiv \ (L_{a*(b*c)}\,+\,
L_bL_aL_c \,+\,
	L_cL_aL_b)d\qquad&&\text{rewriting}.
\end{alignat*}

\ref{jaeqs-3}.\ 
By repeatedly applying the equation for~$L_{a*(b*c)}$
from~\ref{jaeqs-2}
it is clear that~$L_{a^n}$
and~$L_{a^m}$
may both be written as polynomial
expressions
in~$L_{a}$ and $L_{a^2}$.
Since~$L_a$ and~$L_{a^2}$ commute
by the Jordan identity,
so will~$L_{a^n}$ and~$L_{a^m}$ commute.
Whence $a^n *(b* a^m) = (a^n * b)*a^m$.

Finally, seeing that $a^n*a^m=a^{n+m}$ is only a matter of induction over~$m$.
Indeed, $a^n*a^0=a^n*1=a^n$,
and if~$a^n*a^m=a^{n+m}$ for all~$n$ for some fixed~$m$,
we get $a^n*a^{m+1}
= a^n*(a*a^m)=(a^n*a)*a^m=a^{n+1}*a^m=a^{n+m+1}$.
\end{proof}

\begin{corollary}\label{cor:assocalg}
    Let $a\in E$ be an element of an EJA. Let $C(a)$ denote the closure of the algebra generated by $a$, then $C(a)$ is a commutative associative algebra.
\end{corollary}
\begin{proof}
Point~\ref{jaeqs-3} of proposition~\ref{jaeqs}
allows us to see that the smallest Jordan subalgebra of~$E$
that contains~$a$
consists of all real polynomials $\sum_{n=0}^N \lambda_n a^n$ over~$a$,
and is therefore associative.
Since the Jordan multiplication is continuous 
(by proposition~\ref{prod-bounded}) 
the closure $C(a)$ of this associative subalgebra
will again be an associative subalgebra.
\end{proof}

\begin{proposition}\label{prop:associsdisc}
    An associative EJA is isomorphic as an algebra to $\R^n$ with pointwise multiplication for some $n\in \N$.
\end{proposition}
\begin{proof}
    Let $E$ be an associative EJA and let $L_a:E\rightarrow E$ denote the Jordan multiplication operator of $a\in E$. This gives rise to a map $L:E\rightarrow B(E)$ that is linear (since $L_{a+b}=L_a+L_b$ and $L_{\lambda a} = \lambda L_a$), multiplicative (by associativity $L_{a*b} = L_aL_b$), unital ($L_1=\id$), injective (since $L_a1 = a$) and positive ($L_a$ is self-adjoint and $L_{a^2} = L_a^2$ is therefore a positive operator). The map is also order-reflecting. To see this we first note that the algebra $C(L_a)$ generated by $L_a$ in $B(E)$ is equal to the set $L(C(a)):=\{L_b~;~b\in C(a)\}$. Now if $L_a\geq 0$ in $B(E)$, then it has a square root which lies in $C(L_a)=L(C(a))$, so that we can find a $b\in C(a)$ with $L_{b^2}=L_b^2 = \sqrt{L_a}^2 = L_a$ so that $a$ is indeed positive in $E$. We conclude that $E$ is order-isomorphic to some closed subspace of $B(E)$ and thus that $E$ is a complete Archimedean order unit space.

    The product of positive elements is positive, since
    indeed: $a^2*b^2 = (a*a)*(b*b) = (a*b)*(a*b) = (a*b)^2$.
    By Kadison's representation theorem~\cite{kadison1951representation}
        any complete Archimedean order unit space with unital multiplication
        that preserves positivity (like~$E$), is isomorphic as an algebra
        to~$C(X)$, the real-valued continuous functions on some
            compact Hausdorff space~$X$.

    Thus without loss of generality, we may assume~$E = C(X)$.
    It is sufficient to show~$X$ is discrete (for then~$X$
        must be finite by compactness).
    For~$x \in X$, write~$\delta_x \colon C(X) \to \R$ for
        the bounded linear map~$\delta_x(f)= f(x)$.
    As~$E=C(X)$ is assumed to be a Hilbert space,
        there must be an~$\hat{x} \in C(X)$
        with~$\delta_x(f) = \inn{\hat{x},f} = f(x)$ for all~$f \in C(X)$.
    As~$\inn{\hat{x}g,f} = \inn{\hat{x},gf}
    = (gf)(x) = g(x)f(x) = \inn{\hat{x},g}\inn{\hat{x},f} =
    \inn{\inn{\hat{x},g}\hat{x},f}$ for all~$f \in C(X)$,
    we must have~$\hat{x}g = \inn{\hat{x},g}\hat{x}$.
    In particular~$\hat{x}\hat{y} = \inn{\hat{x},\hat{y}}\hat{x}$
    and with similar reasoning~$\hat{x}\hat{y} = \inn{\hat{x},\hat{y}}\hat{y}$.
    Assume~$x\neq y$.
    Then~$\hat{x} \neq \hat{y}$, but by the
        previous~$\inn{\hat{x},\hat{y}} \hat{x} =  \inn{\hat{x},\hat{y}} \hat{y}$.
        So that necessarily~$0 = \inn{\hat{x},\hat{y}}=\hat{x}(y)$ for all~$y \neq x$.
        As~$\hat{x} \neq 0$ and~$\hat{x}$ is continuous,
            we see~$\{x\}$ is open and so~$X$ is discrete.
\end{proof}

\begin{corollary}\label{cor:spectral}
    Let $a$ be an element of an EJA. Then there exist real numbers $\lambda_i$ and orthogonal idempotents $p_i$ such that $a = \sum_{i=1}^n \lambda_i p_i$ for some $n$.
\end{corollary}
\begin{proof}
    Let $C(a)$ denote the EJA generated by $a$. This is an associative algebra by corollary \ref{cor:assocalg} so that by proposition \ref{prop:associsdisc} we have $C(a)\cong \R^n$ for some $n$. Since $\R^n$ is obviously spanned by orthogonal idempotents we see that indeed $a= \sum_{i=1}^n \lambda_i p_i$.
\end{proof}

\begin{proposition}\label{prop:selfduality}
    An element $a\in E$ is positive (i.e.\ a square) if and only if $\inn{a,b}\geq 0$ for all positive $b$.
\end{proposition}
\begin{proof}
    If $p$ is an idempotent then $\inn{p,a}\geq 0$ if $a$ is positive \cite[p. 107]{chu2011jordan}. As a result if $b=\sum_i \lambda_i p_i$ with $\lambda_i \geq 0$ and with the $p_i$ idempotents we have $\inn{a,b}\geq 0$. Now for the other direction suppose $\inn{a,b}\geq 0$ for all positive $b$. Write $a=\sum_i \lambda_i p_i$ with the $\lambda_i$ not necessary positive, and where the $p_i$ are orthogonal. We then have $\inn{p_i,p_j} = \inn{1,p_i*p_j} = \inn{1,0} = 0$ so that $0\leq \inn{a,p_j} = \lambda_j \inn{p_j,p_j}$. Since $p_j\neq 0$ this is only possible when $\lambda_j\geq 0$. This holds for all $j$ so that we conclude that $a\geq 0$.
\end{proof}
\begin{corollary}\label{cor:ejaous}
    Let $E$ be an EJA. The set of positive elements is closed under addition. More specifically $E$ is an Archimedean order unit space.
\end{corollary}
\begin{proof}
    By the previous proposition $a\geq 0$ if and only if $\inn{a,b}\geq 0$ for all $b\geq 0$. But then if $c\geq 0$ we obviously have $\inn{a+c,b} = \inn{a,b}+\inn{c,b} \geq 0$ for all positive $b$ so that indeed $a+c\geq 0$. Suppose now that $a\leq \frac{1}{n}1$ for all $n\in \N$. By proposition \ref{prop:associsdisc} the associative algebra generated by $a$ (which contains $\frac{1}{n}1$) is isomorphic to $\R^n$. Since this space is Archimedean we conclude that $a\leq 0$ in $\R^n$ so that also $a\leq 0$ in $E$. In the same way we can find for any $a\in E$ a number $n\in \N$ so that $-n1 \leq a\leq n1$ so that $E$ is indeed an Archimedean order unit space.
\end{proof}

\begin{proposition}\label{prop:equivalenttopology}
    Let $E$ be an EJA. The topologies induced by the Hilbert norm and by the order unit norm are equivalent.
\end{proposition}
\begin{proof}
    In order to show that the topologies are the same we need to show that the norms are equivalent. Let $\norm{a}$ denote the order unit norm and $\norm{a}_2$ the Hilbert norm. We need to find constants $c,d\in \R_{>0}$ such that $c\norm{a}_2\leq \norm{a}\leq d\norm{a}_2$ for all $a\in E$.

    Note that $\norm{a}_2^2 = \inn{a,a}\leq \norm{a}^2 \inn{1,1} = \norm{a}^2\norm{1}_2^2$ by self-duality, so that we already have one side of the inequality.

    Any $a\in E$ can be written as $a=\sum_i \lambda_i p_i$ where the $p_i$ are nonzero by the previous corollary so that $\norm{a}=\max\{\lvert \lambda_i \rvert\}$. Now $\norm{a}_2^2 = \inn{a,a} = \sum_i \lambda_i^2 \norm{p_i}_2^2 \geq \sum_i \lambda_i^2 \inf\{\norm{p_j}_2^2\} \geq \max\{\lvert \lambda_i^2 \rvert\} \inf\{\norm{p_j}_2^2\} = \norm{a}^2 \inf\{\norm{p_j}_2^2\}$ so if we can find some constant $R>0$ such that $\norm{p}_2 \geq R$ for all nonzero idempotents $p$ we are done.

    Let $R=\inf_{p\neq 0, p^2=p} \norm{p}_2$. If $R\neq 0$ we are done, so suppose $R=0$. In this case there exists a sequence of idempotents $(p_i)$ such that $\norm{p_i}_2\rightarrow 0$. We can then pick a subsequence such that $\norm{p_k}_2 \leq 2^{-k}/k$. Now let $q_n = \sum_{i=1}^n ip_i$. Let $n\geq m$. We have $\norm{q_n - q_m}_2 = \norm{\sum_{k=m}^n k p_k}_2 \leq \sum_{k=m}^n k\norm{p_k}_2 \leq \sum_{k=m}^n k 2^{-k}/k$ so that the $(q_n)$ form a Cauchy sequence in the Hilbert norm. Since $E$ is a Hilbert space it must converge to some $q\in E$ and since it is also an increasing sequence and the set of positive elements is closed in the Hilbert norm by proposition \ref{prop:selfduality} we must have $q\geq q_n$ so that $\norm{q}\geq \norm{q_n} \geq n$ for all $n$ which is a contradiction.
\end{proof}

\begin{proposition}\label{prop:EJAisJB}
    Let $E$ be an EJA. Then $E$ is a JB-algebra.
\end{proposition}
\begin{proof}
    By corollary \ref{cor:ejaous} $E$ is an Archimedean order unit space. By definition $E$ is complete in the Hilbert norm topology and by the previous proposition this topology is equivalent to the order unit topology. We conclude that $E$ is a complete Archimedean order unit space. By \cite[Theorem 1.11]{alfsen2012geometry}, $E$ will then be a JB-algebra when the implication $-1\leq a\leq 1 \implies 0\leq a^2\leq 1$ holds. So suppose $-1\leq a\leq 1$. By the spectral theorem $a=\sum_i \lambda_i p_i$ and we must have $-1\leq \lambda_i \leq 1$. But then $a^2 = \sum_i \lambda_i^2 p_i$ so that indeed $0\leq a^2 \leq 1$.
\end{proof}

\begin{proposition}\label{prop:dircomplete}
    Let $E$ be an EJA. Then $E$ is bounded directed complete and furthermore every state is normal.
\end{proposition}
\begin{proof}
    Let $(a_i)_{i\in I}$ be a bounded upwards directed set. By translation we can take all $a_i$ to be positive. Define for $b\geq 0$ the state $\omega(b):= \sup_{i\in I} \inn{a_i,b}$. This supremum exists since the $a_i$ are bounded and the inner product between positive elements is again positive. This map can obviously be extended by linearity to the entirety of $E$. Since $E$ is a Hilbert space we conclude that there must exist an $a\in E$ such that $\omega(b)=\inn{a,b}$ for all $b\in E$. We claim that this $a$ is the lowest upper bound. That it is an upper bound follows by the self-duality of the order. Suppose $a_i\leq c$ for some $c$. Then $c-a_i\geq 0$ so that $\inn{c-a_i,b}\geq 0$ for all $b\geq 0$ so that $\inn{c,b}\geq \inn{a_i,b}$. By taking the supremum over the $a_i$'s we see then that $\inn{c,b}\geq \inn{a,b}$. Again by self-duality we conclude that $c\geq a$.

    For any state $\omega^\prime: E\rightarrow \R$ we can find a $b\in E$ such that $\omega^\prime = \inn{\cdot, b}$. As the previous argument shows, suprema of elements are defined in terms of these states so that they must preserve those suprema.
\end{proof}

\begin{proposition}\label{prop:atomicsum}
    Let $p$ be an idempotent of an EJA. Then there exist orthogonal atomic idempotents $p_i$ such that $p = \sum_i p_i$.
\end{proposition}
\begin{proof}
	If $p$ is atomic we are done, so suppose it is not. Then by definition we can find $0\leq a \leq p$ such that $a\neq \lambda p$ for some $\lambda$. Using corollary~\ref{cor:spectral} write $a=\sum_i \lambda_i q_i$. If all the $q_i = p$ then $a=\lambda p$ so there must be a $q_i\neq p$. Pick this one. We have $\lambda_i q_i \leq p$. By proposition \ref{prop:EJAidempotent} we then have $Q_p (\lambda_i q_i) = \lambda_i q_i$. This of course implies $Q_p q_i = q_i$ so that again by proposition \ref{prop:EJAidempotent} we have $q_i\leq p$. We can now repeat this procedure with $p$ replaced by $q_i$ and $p-q_i$ to get a family of orthogonal idempotents that sum up to $p$. We claim that this process stops after a finite amount of iterations. By assumption the resulting idempotents are then atomic.

    Suppose the process does not halt after a finite amount of iterations. Then we are left with a countable collection of orthogonal idempotents $(q_i)_i$. By equation \eqref{jaeqs-1} we have in any Jordan algebra for any $a$ and $b$: $[L_a,L_{b^2}] + 2[L_b, L_{a*b}] = 0$ Let $a=q_i$ and $b=q_j$ with $i\neq j$ so that $a*b=0$, then we conclude that $[L_{q_i},L_{q_j}] = 0$ and thus that $q_i*(a*q_j) = (q_i*a)*q_j$. The algebra spanned by the $(q_i)_i$ is therefore associative. As this algebra is necessarily infinite-dimensional this is in contradiction to proposition \ref{prop:associsdisc}.
\end{proof}

\begin{proposition}
    Let $E$ be an EJA. Then $E$ is a type I JBW-algebra of finite rank.
\end{proposition}
\begin{proof}
    By proposition \ref{prop:EJAisJB} $E$ is a JB-algebra. By definition, a JBW-algebra is a JB-algebra that is bounded directed complete and that is separated by normal states. Proposition \ref{prop:dircomplete} therefore has established that $E$ is indeed a JBW-algebra. A JBW-algebra is of type I if below every idempotent we can find an atomic idempotent. This is true by proposition \ref{prop:atomicsum}. By this same proposition we can write the identity as a finite sum of atomic idempotents, so that the space is indeed of finite rank.
\end{proof}

\begin{corollary}\label{appendixjnw}
    Let $E$ be an EJA. Then there exists a finite-dimensional EJA $E_{\text{fin}}$ and an EJA $E_{\text{inf}}$ that is a direct sum of infinite-dimensional spin-factors such that $E$ is isomorphic as an EJA to $E_{\text{fin}} \oplus E_{\text{inf}}$.
\end{corollary}
\begin{proof}
    By the previous corollary $E$ is a type I JBW-algebra of finite rank. It is therefore isomorphic to a finite direct sum of type I JBW-factors of finite rank. These factors have been classified in \cite{hanche1984jordan}. They are either finite-dimensional or they are infinite-dimensional spin-factors.
\end{proof}

\end{document}